\documentclass[12pt]{amsart}

\usepackage[english]{babel}
\usepackage[utf8x]{inputenc}
\usepackage[T1]{fontenc}
\usepackage{todonotes}
\usepackage{tikz}
\usepackage{mathtools}

\usepackage[a4paper,top=3cm,bottom=3cm,left=3cm,right=3cm,marginparwidth=1.75cm]{geometry}

\usepackage{varioref}
\usepackage{hyperref}
\usepackage[nameinlink, capitalize, noabbrev]{cleveref}

\usepackage{graphicx}
\usepackage{amsmath}
\theoremstyle{plain} 
\newtheorem{thmx}{Theorem}

\newtheorem{thm}{Theorem}[section]
\newtheorem{lemma}[thm]{Lemma}
\newtheorem{prop}[thm]{Proposition}
\newtheorem{corl}[thm]{Corollary}

\theoremstyle{definition} 
\newtheorem{defn}[thm]{Definition}

\theoremstyle{remark} 
\newtheorem{rmk}[thm]{Remark}
\numberwithin{equation}{section}

\DeclareMathOperator{\Id}{Id}     

\newcommand{\A}{\mathcal{A}}		
\newcommand{\B}{\mathcal{B}}		
\newcommand{\R}{\mathbb{R}}			
\newcommand{\C}{\mathbb{C}}			
\newcommand{\Z}{\mathbb{Z}}			
\newcommand{\N}{\mathbb{N}}			
\newcommand{\T}{\mathbb{T}}			
\newcommand{\U}{\mathcal{U}}        

\newcommand{\GS}{\mathcal{G}}		    
\newcommand{\Sub}{\Delta}     			
\newcommand{\Subo}{\Delta^{\circ}}		
\newcommand{\dif}{\mathrm{d}}           
\newcommand{\cbar}{\overline{c}}

\def\hs#1#2{\left\langle #1,#2\right\rangle} 

\def\tfp#1{#1 \times \widehat{#1}} 

\title[Spectral invariance of twisted convolution algebras]{Spectral invariance of $*$-representations of twisted convolution algebras with applications in Gabor analysis}
\author{Are Austad}
\date{\today}
\address{Norwegian University of Science and Technology, Department of Mathematical Sciences, Trondheim, Norway.}
\email{are.austad@ntnu.no}

\begin{document}
\maketitle

\begin{abstract}
We show spectral invariance for faithful $*$-representations for a class of twisted convolution algebras. More precisely, if $G$ is a locally compact group with a continuous $2$-cocycle $c$ for which the corresponding Mackey group $G_c$ is $C^*$-unique and symmetric, then the twisted convolution algebra $L^1 (G,c)$ is spectrally invariant in $\mathbb{B}(\mathcal{H})$ for any faithful $*$-representation of $L^1 (G,c)$ as bounded operators on a Hilbert space $\mathcal{H}$. As an application of this result we give a proof of the statement that if $\Delta$ is a closed cocompact subgroup of the phase space of a locally compact abelian group $G'$, and if $g$ is some function in the Feichtinger algebra $S_0 (G')$ that generates a Gabor frame for $L^2 (G')$ over $\Delta$, then both the canonical dual atom and the canonical tight atom associated to $g$ are also in $S_0 (G')$. We do this without the use of periodization techniques from Gabor analysis.
\end{abstract}

\section{Introduction}
The primary focus of this article is the concept of spectral invariance. In short, if $\A$ is a $*$-subalgebra of a Banach $*$-algebra $\B$, then $\A$ is said to be \emph{spectrally invariant} in $\B$ if $\sigma_{\A} (a) = \sigma_{\B}(a)$ for all $a \in \A$, where $\sigma_{\A}(a)$ denotes the spectrum of the element $a$ in the algebra $\A$, and likewise for $\sigma_{\B}(a)$. In particular, if $\A$ and $\B$ are both unital with common unit, and if $a\in \A$ is invertible in $\B$, spectral invariance of $\A$ in $\B$ tells us that $a^{-1} \in \A$ as well. Spectral invariance of Banach $*$-algebras in $C^*$-algebras is a concept that has been extensively studied and is of importance in a number of different mathematical fields. Due to the seminal paper \cite{Gelfand39-1} the study of spectral invariance has been linked to Wiener's lemma, and variations of this result. As fields where spectral invariance is of importance we mention the theory of noncommutative tori \cite{Connes80, grle04}, Gabor analysis and window design in the theory of Gabor frames \cite{grle04}, convolution operators on locally compact groups \cite{Barnes90, FenGrLe06, FenGrLeLu03}, infinite-dimensional matrices \cite{Baskakov90, GoKaWo89, Kurbatov90, Sjostrand95}, and the theory of pseudodifferential operators \cite{Gr06pedestrian, Gr06sjostrand, GrSt07, Sjostrand95}. This list is by no means exhaustive. For an introduction to these variations on spectral invariance and Wiener's lemma we refer the reader to \cite{Gr10wiener}.

The main motivations for this article are the uses of spectral invariance in noncommutative geometry \cite{co94-1} and in Gabor analysis \cite{grle04}. Indeed, the original motivation for this article was to prove an extension of the main result of the latter article in the case of closed cocompact subgroups of the phase space of a locally compact abelian group without using periodization techniques from Gabor analysis. We do this in \cref{sec:application-Gabor}. Our focus will not be on general $*$-subalgebras of Banach $*$-algebras. Instead we will limit ourselves to a subclass of all twisted convolution algebras of locally compact groups where the twist is implemented by a continuous $2$-cocycle, see \cref{def:2-cocycle}. Given such a locally compact group $G$ and a continuous $2$-cocycle $c$, the resulting twisted convolution algebra will be denoted $L^1 (G,c)$. Given a faithful $*$-representation $\pi \colon L^1 (G,c) \to \mathbb{B}(\mathcal{H})$ for some Hilbert space $\mathcal{H}$, we wish to find conditions on $G$ and $\pi$ that guarantee that $\sigma_{L^1 (G,c)}(f) = \sigma_{\mathbb{B}(\mathcal{H})}(\pi (f))$ for all $f \in L^1 (G,c)$, i.e.\ that $L^1 (G,c)$ is spectrally invariant in $\mathbb{B}(\mathcal{H})$. Key to our approach to this problem is the use of the Mackey group $G_c$ associated to the locally compact group $G$ and the continuous $2$-cocycle $c$, and we define this group in \cref{subsec:obtaining-C*-algs}. Note that in general $L^1 (G,c)$ and $L^1 (G_c)$ are not isomorphic as Banach $*$-algebras. It will be of importance to us that the convolution algebra $L^1 (G_c)$ is symmetric, which in short means that the positive elements of the Banach $*$-algebra $L^1 (G_c)$ have positive spectra, see \cref{defn:symmetric-banach-alg}. We then  apply a result of Hulanicki \cite{Hula72}, stated for the reader's convenience in \cref{thm:Hulanicki-result}, to prove prove the main result of the article. 


Due to the use of the result of Hulanicki, the argument for spectral invariance will depend on a norm condition on self-adjoint elements. This norm condition may be difficult to check in practice, so we describe a class of groups for which the condition is automatically satisfied. This leads us to $C^*$-unique groups, introduced by Boidol \cite{Boidol84}. In short, a locally compact group $G$ is $C^*$-unique if its convolution algebra $L^1 (G)$ has a unique $C^*$-norm. A Banach $*$-algebra admitting a faithful $*$-representation is called $C^*$-unique if it has a unique $C^*$-completion. Examples of $C^*$-unique groups are semidirect products of abelian groups, connected metabelian groups, as well as groups where every compactly generated subgroup is of polynomial growth \cite[p.\ 224]{Boidol84}. We may now state the article's main theorem. 

\begin{thmx}[\cref{thm:main-thm}]\label{thm:TheoremA}
	Let $G$ be a locally compact group with a continuous $2$-cocycle $c$.
	\begin{enumerate}
	    \item [i)] If $L^1 (G_c)$ is $C^*$-unique, so is $L^1 (G,c)$. 
	    \item [ii)] If $L^1 (G_c)$ is symmetric and $C^*$-unique and $\pi \colon L^1 (G,c)\to \mathbb{B}(\mathcal{H})$ is a faithful $*$-representation, then $f \mapsto \Vert \pi (f)\Vert_{\mathbb{B}(\mathcal{H})}$, $f\in L^1 (G,c)$, is the full $C^*$-norm on $L^1 (G,c)$, and $\sigma_{L^1 (G,c)}(f) = \sigma_{\mathbb{B}(\mathcal{H})}(\pi(f))$ for all $f \in L^1 (G,c)$.
	\end{enumerate}
\end{thmx}
Though there are some known examples of $C^*$-unique groups, there are very few statements in the literature concerning the $C^*$-uniqueness of twisted convolution algebras. This is why we go via the convolution algebra of the Mackey group $G_c$, and why statement i) is of independent interest. Note also that for all unital Banach $*$-algebras, being symmetric is equivalent to being spectrally invariant in the enveloping $C^*$-algebra, see for example \cite[p.\ 340]{Li12}. 

Important to our proof of the main theorem is the observation that convolution in $L^1 (G_c)$ can be expressed in terms of convolution in the algebras $L^1 (G,c^n)$, $n \in \Z$, where $c^n$ is the $2$-cocycle $c$ raised to the nth power, see \cref{prop:convolution-and-involution-identities}. As an immediate consequence, $L^1 (G_c)$ can be decomposed in terms of the subalgebras $L^1 (G,c^n)$ as in \cref{corl:graded-subspace-iso}, and this allows us to extend a faithful $*$-representation of $L^1 (G,c)$ to a faithful $*$-representation of $L^1 (G_c)$ in the proof of \cref{thm:main-thm}. This is truly the crucial step in the proof. 

Using our main theorem we are able to give a short proof on a problem concerning regularity of canonical dual atoms and canonical tight atoms in Gabor analysis. We will do this by restating the problem in operator algebraic terms and then use \cref{thm:main-thm}. Exploring the interplay between Gabor analysis and operator algebras has gained much popularity in recent years \cite{AuEn20, AuJaLu19, Enstad19, jalu18duality, Kreisel16, lu09, lu11}. The field of Gabor analysis has its origins in the seminal paper of Gabor \cite{Ga46}, where he claimed that it is possible to obtain basis-like representations of functions in $L^2(\R)$ in terms of the set $\{ e^{2\pi i lx} \phi(x- k) : k,l \in \Z \}$, where $\phi$ denotes a Gaussian. A central problem of the field is still to find basis-like expansions of functions in terms of time-frequency shifts of the form \eqref{eq:tf-shift}. Although most research in this field is done on one or several real variables, it is possible, due to the nature of time-frequency shifts, to study Gabor analysis on phase spaces of locally compact abelian groups \cite{gr98}. Let $G$ be a locally compact abelian group. Then its phase space is the group $\tfp{G}$, where $\widehat{G}$ is its Pontryagin dual. Let $\pi (z)$ be a time-frequency shift of the form \eqref{eq:tf-shift} for some point $z = (x,\omega)\in \tfp{G}$. Ignoring normalizations on the relevant Haar measures for the time being, one may then consider a closed, cocompact subgroup $\Sub \subseteq \tfp{G}$ and a function $g \in L^2 (G)$ and ask when a set 
$\GS(g;\Sub) :=(\pi(z)g)_{z\in \Sub}$ is a frame for $L^2 (G)$, i.e.\ when there exist constants $C,D>0$ for which
\begin{equation*}
    C\Vert f\Vert_2^2 \leq \int_{\Sub} \vert \hs{f}{\pi(z)g}\vert^2 \dif z \leq D \Vert f \Vert_2^2
\end{equation*}
holds for all $f \in L^2 (G)$, where $\dif z$ is the chosen Haar measure on $\Sub$. The reason for assuming that $\Sub$ is cocompact will be explained in \cref{rmk:cocompact-necessary}. In time-frequency analysis it is often also of interest that the Gabor atom $g$ has good time-frequency decay. One way of expressing good time-frequency decay is to say that $g$ is in Feichtinger's algebra $S_0 (G)$, see \eqref{eq:def-feichtinger-alg}. 

Equivalent to $\GS(g;\Sub)$ being a Gabor frame for $L^2 (G)$ is the invertibility of the frame operator $S\colon L^2 (G) \to L^2 (G)$ associated to $\GS(g;\Sub)$. The form of $S$ most suitable for our purposes is given in \eqref{eq:frame-op-rewritten}. Two functions of interest are then the canonical dual atom of $g$, which is $S^{-1}g$, and the canonical tight atom associated to $g$, which is $S^{-1/2}g$. They are of importance in Gabor analysis since they allow for perfect reconstuction formulas for all functions in $L^2 (G)$ in terms of $g, S^{-1}g$, and $S^{-1/2}g$, as illustrated by \eqref{eq:dual-atom-reconstruction} and \eqref{eq:tight-atom-reconstruction}. If $g\in S_0 (G)$ generates a frame $\GS (g;\Sub)$ for $L^2 (G)$, a natural question in Gabor analysis is then whether $S^{-1}g$ and $S^{-1/2}g$ are in $S_0 (G)$ also. This leads us to our second main result.
\begin{thmx}[\cref{thm:continuous-frame-inverse-regularity}]
	Let $\Sub \subseteq \tfp{G}$ be a closed cocompact subgroup, and suppose $g \in S_0 (G)$ is such that $\GS(g;\Sub)$ is a Gabor frame for $L^2 (G)$. Then $S^{-1}g, S^{-1/2}g \in S_0 (G)$ as well. 
\end{thmx}
We note that the above result was proved in the case of separable lattices in $\R^{2d}$, and claimed to hold more generally for lattices in phase spaces of locally compact abelian groups, in the celebrated paper \cite{grle04}. Though it is somewhat technical to prove \cref{thm:main-thm}, our approach to \cref{thm:continuous-frame-inverse-regularity} presented below makes it simple to prove the extension of the main result of \cite{grle04} for general closed cocompact subgroups rather than just lattices. It may be possible to adapt the proof from \cite{grle04} to this setting as well, but we offer an independent proof which makes no use of periodization techniques available in the setting of Gabor analysis.

As mentioned, to prove \cref{thm:continuous-frame-inverse-regularity} we will restate the problem in operator algebraic language. For a Gabor frame $\GS (g;\Sub)$ with $g \in S_0 (G)$, the frame operator $S$ can be rephrased in terms of a faithful (right) $*$-representation of the Banach $*$-algebra $\ell^1 (\Subo, \cbar)$, where $\Subo$ is the adjoint lattice of $\Sub$ and $c$ is the Heisenberg $2$-cocycle, see \eqref{eq:adjoint-subgroup} and \eqref{eq:heisenberg-cocycle}. As we explain in the proof of \cref{thm:continuous-frame-inverse-regularity}, locally compact abelian groups are $C^*$-unique and for any continuous $2$-cocycle on them the associated Mackey group $G_c$ is also $C^*$-unique. In addition, $L^1 (G_c) $ will in this case be symmetric. Hence we may apply \cref{thm:main-thm} to obtain our second main result.

The article is organized as follows. \cref{sec:twisted-conv-algs} is dedicated to revising some results on how we obtain twisted convolution algebras and $C^*$-algebras through projective unitary representations of locally compact groups, as well as some results on symmetric convolution algebras and $C^*$-unique groups. Our first main result is \cref{thm:main-thm}, and most of \cref{sec:spec-inv} is dedicated to the proof of this theorem, though some results are of independent interest. In \cref{sec:application-Gabor} we rephrase a problem in Gabor analysis in terms of a faithful $*$-representation of a twisted convolution algebra, and apply \cref{thm:main-thm} to obtain a simple proof of the main result of this section, \cref{thm:continuous-frame-inverse-regularity}.

\section{Twisted convolution algebras}\label{sec:twisted-conv-algs}
\subsection{Projective unitary representations and twisted convolution algebras}\label{subsec:obtaining-C*-algs}
We dedicate this section to explaining how we obtain twisted convolution algebras from projective unitary representations of locally compact groups. Associated to a given projective unitary representation is a continuous $2$-cocycle, defined as follows. 

\begin{defn}\label{def:2-cocycle}
Let $G$ be a locally compact group. By a \emph{continuous $2$-cocycle} for $G$ we mean a continuous map $c : G \times G \to \T$ satisfying
\begin{equation}\label{eq:2-cocycle-assoc}
        c (x_1 , x_2) c(x_1 x_2 , x_3) = c(x_1, x_2 x_3) c(x_2, x_3)
\end{equation}
and
\begin{equation}\label{eq:2-cocycle-unit-prop}
    c(x,e) = c(e,x) = 1
\end{equation}
for all $x,x_1,x_2,x_3 \in G$, and where $e$ is the unit of $G$.
\end{defn}

For completeness we record the following results which follow directly from \cref{def:2-cocycle}.
\begin{lemma}\label{lemma:useful-cocycle-props}
For a continuous $2$-cocycle $c$ for a locally compact group $G$ we have
\begin{enumerate}
    \item [i)] For any $n \in \Z$, the map $c^n : G \times G \to \T$ given by
    \begin{equation*}
        c^n(x_1,x_2) = (c(x_1,x_2))^n, \quad x_1,x_2 \in G,
    \end{equation*}
    is also a continuous $2$-cocycle.
    \item [ii)] For all $x \in G$ we have
    \begin{equation*}\label{eq:2-cocycle-elts-and-inverse}
    c(x,x^{-1}) = c(x^{-1},x).
    \end{equation*}
    \item [iii)] For all $x,y \in G$ we have
    \begin{equation}\label{eq:2-cocycle-reduction}
        c(y,y^{-1}) \overline{c(y^{-1},x)} = c(y,y^{-1}x).
    \end{equation}
\end{enumerate}
\end{lemma}
\begin{proof}
Statement i) is obvious. Statement ii) follows by setting $x_1=x_3 = x$ and $x_2 = x^{-1}$ in \eqref{eq:2-cocycle-assoc} and then using \eqref{eq:2-cocycle-unit-prop}. For iii) we may equivalently show that 
\begin{equation*}
    c(y,y^{-1}) c(yy^{-1}, x) = c(y,y^{-1}x) c(y^{-1},x)
\end{equation*}
since $c(yy^{-1}, x)=1$. Setting $x_1 = y$, $x_2 = y^{-1}$ and $x_3 = x$ in \eqref{eq:2-cocycle-assoc} and then using \eqref{eq:2-cocycle-unit-prop} we obtain the result.
\end{proof}
The continuous $2$-cocycles defined above are indeed part of a cohomology theory for groups. 
This is however not something we shall have much use for in the sequel.

As alluded to above, the notion of continuous $2$-cocycles comes up very naturally when talking about projective unitary representations of groups. A \emph{projective unitary representation} of a locally compact group $G$ is a strongly continuous map $\pi \colon G \to \U(\mathcal{H})$ satisfying
\begin{equation*}
\begin{split}
    \pi(e) &= \Id_{\mathcal{H}}\\
    \pi (x_1) \pi (x_2) &= c(x_1,x_2) \pi (x_1 x_2)
\end{split}
\end{equation*}
for $x_1, x_2 \in G$, where $e$ is the unit of $G$, and where $c : G \times G \to \T$ is a priori just some continuous map. Here $\U(\mathcal{H})$ denotes the unitary operators on some Hilbert space $\mathcal{H}$. Associativity then yields that $c$ must satisfy \eqref{eq:2-cocycle-assoc}, and since $\pi (e) = \Id_{\mathcal{H}}$, we also get \eqref{eq:2-cocycle-unit-prop}. 
Hence $c$ is a continuous $2$-cocycle for $G$. A projective unitary representation of $G$ where the projectivity is governed by $c$ as above will be called a \emph{$c$-projective unitary representation of $G$}. If $c =1$ we just call them unitary representations. 

Given a locally compact group $G$ and a continuous $2$-cocycle $c$, we can construct an associated group $G_c$. As a topological space, $G_c$ is just $G \times \T$ with the product topology (where $\T$ has its usual topology and group structure as $\U(\C)$). However, the group structure is not in general the simple direct product of the respective group structures. Instead we set
\begin{equation}\label{eq:mackey-group-law}
    (x_1, \tau_1) (x_2, \tau_2) = (x_1 x_2 , \tau_1 \tau_2 \overline{c(x_1,x_2)}).
\end{equation}
This is indeed a group. The identity is given by $(e,1)$, and the inverse of an element $(x,\tau) \in G_c$ is given by
\begin{equation*}\label{eq:mackey-inverse}
    (x,\tau)^{-1} = (x^{-1} , \overline{\tau} c(x^{-1},x)).
\end{equation*}
With the product topology and group law of \eqref{eq:mackey-group-law} $G_c$ becomes a locally compact group we call the \emph{Mackey group}. It is well known that the resulting Haar measure is just the product measure of the measures on $G$ and $\T$, and hence the modular function on $G_c$ may be identified with the modular function on $G$. We will in the sequel normalize the measure on $\T$ such that $\mu (\T)=1$, where $\mu$ is the Haar measure.

The usefulness of the Mackey group for us is in the fact that $c$-projective unitary representations of $G$ induce unitary representations of $G_c$. This is done by sending a $c$-projective unitary representation of $G$, say $\pi \colon G \to \U (\mathcal{H})$ for some Hilbert space $\mathcal{H}$, to $\pi_c \colon G_c \to \U (\mathcal{H})$, where
\begin{equation}\label{eq:induced-rep-on-Mackey}
\pi_c (x,\tau) = \overline{\tau} \pi (x)
\end{equation}
for $(x,\tau)\in G_c$.

Given a locally compact group $G$ and a continuous $2$-cocycle $c$ for $G$, there is always a distinguished $c$-projective unitary representation of $G$. Let $L^2 (G)$ denote the square-integrable measurable functions on $G$. The distinguished $c$-projective unitary representation is called the \emph{$c$-twisted left regular representation of $G$}, and it is the map $L^c \colon G \to \U( L^2 (G))$ given by
\begin{equation*}\label{eq:twisted-left-regular}
	L^c_y f(x) = c(y,y^{-1}x) f(y^{-1}x), \quad \text{$x,y\in G$, $f \in L^2(G)$.}
\end{equation*}
If $c=1$ we drop the $c$ from the notation and just write $L_y$ for $y \in G$. 

We proceed to introduce twisted convolution algebras of these groups and show how we may complete them to $C^*$-algebras. For a locally compact group $G$ with modular function $m$, we consider the space of measurable and integrable functions $L^1 (G)$. For a continuous $2$-cocycle $c$ for $G$ we define $c$-twisted convolution on $L^1 (G)$ by
\begin{equation*}\label{eq:twisted-conv}
    f_1 \natural_c f_2 (x) = \int_G f_1 (y) f_2 (y^{-1} x) c(y,y^{-1}x) 
    \dif y,
\end{equation*}
for $f_1,f_2 \in L^1 (G)$, where $\dif y$ is the Haar measure on $G$. Should $f_2 \in L^p (G)$, $p\in [1,\infty]$ we will use the same notation. We also define the $c$-twisted involution 
\begin{equation*}\label{eq:twisted-inv}
    f^{*_c}(x) = m(x^{-1}) \overline{c(x,x^{-1}) f(x^{-1})}
\end{equation*}
for $f \in L^1 (G)$. With these operations $(L^1 (G), \natural_c, {}^{*_c})$ becomes a Banach $*$-algebra with the usual $\Vert \cdot \Vert_1$-norm on $L^1 (G)$. From now on we will just write $L^1 (G,c)$ instead of $(L^1 (G), \natural_c, {}^{*_c})$ to ease notation.

Any $c$-projective unitary representation $\pi \colon G \to \U(\mathcal{H})$ now induces a nondegenerate $*$-representation of the Banach $*$-algebra $L^1 (G,c)$ as bounded operators on $\mathcal{H}$. By slight abuse of notation we will denote the induced $*$-representation by $\pi$ also. For $f \in L^1 (G)$ and $\eta \in \mathcal{H}$ we have
\begin{equation*}\label{eq:L1-integrated-rep}
    \pi (f) \eta = \int_G f(x) \pi (x)\eta \dif x ,
\end{equation*}
where we interpret the integral weakly. Note that $\Vert \pi (f) \Vert \leq \Vert f \Vert_{L^1 (G)}$. If the integrated representation $\pi$ is faithful this gives us a way of completing $L^1 (G,c)$ to a $C^*$-algebra, namely for any $f \in L^1 (G)$ we set
\begin{equation*}\label{eq:C*-norm-completion}
    \Vert f \Vert := \Vert \pi (f) \Vert_{\mathbb{B}(\mathcal{H})}.
\end{equation*}
The integrated representation of the $c$-twisted left regular representation will be denoted by $f \mapsto L^c_f$. 
The following result, which will be important for us in the proof of \cref{thm:main-thm}, is a special case of \cite[Satz 6]{Leptin68}.
\begin{prop}\label{prop:Leptin-twisted-result}
    Let $G$ be an amenable locally compact group with a continuous $2$-cocycle $c$. Then $f \mapsto \Vert L^c_f \Vert_{\mathbb{B}(L^2 (G))}$ is the maximal $C^*$-norm on $L^1 (G,c)$. 
\end{prop}

Instead of twisting the convolution algebra of the locally compact group $G$ by a continuous $2$-cocycle $c$, we could first ''twist'' the group $G$ by $c$ to obtain the associated Mackey group $G_c$, and then consider the associated convolution algebra. That is, we consider the space $L^1 (G_c)$ with convolution
\begin{equation*}
\begin{split}
    F_1 * F_2 (x,\tau) &= \int_{G_c} F_1 ((y,\xi)) F_2 ((y, \xi)^{-1} (x,\tau)) \dif \mu_{G_c} \\
    &= \int_G \int_{\T} F_1 ((y,\xi)) F_2 ((y^{-1}x, \overline{\xi}\tau c(y^{-1},y) \overline{c(y^{-1},x)})) \dif \xi \dif y
\end{split}
\end{equation*}
for $F_1, F_2 \in L^1 (G_c)$ and $(x,\tau)\in G_c$. The involution is given by
\begin{equation*}
    F^* (x,\tau) = m(x^{-1}) \overline{F((x,\tau)^{-1})} = m(x^{-1}) \overline{F((x^{-1}, \overline{\tau} c(x^{-1}, x))}
\end{equation*}
for $F \in L^1 (G_c)$ and $(x,\tau)\in G_c$, and where $m$ is the modular function for $G$. Any $c$-projective unitary representation of $G$ induces a unitary representation $\pi_c$ of $G_c$ by \eqref{eq:induced-rep-on-Mackey}, which in turn induces a nondegenerate $*$-representation $\pi_c$ of $L^1 (G_c)$. Note however that $\pi_c$ is in general not a faithful $*$-representation of $L^1 (G_c)$ even if $\pi_c$ is a faithful unitary representation of $G_c$. Indeed, let $f\in L^1 (G)\setminus \{0\}$ and define $F \in L^1 (G_c)$ by $F(x,\tau) = \overline{\tau}f(x)$. Then
\begin{equation*}
\begin{split}
    \pi_c (F) \eta &= \int_G \int_\T F(x,\tau) \pi_c (x,\tau) \eta \dif \tau \dif x \\
    &= \int_G \int_\T \overline{\tau} f(x) \overline{\tau} \pi(x) \eta \dif \tau \dif x 
    = \int_G \int_\T \overline{\tau}^2 f(x) \pi(x) \eta \dif \tau \dif x = 0,
\end{split}
\end{equation*}
for all $\eta \in \mathcal{H}$, even though $F$ is not the zero function.
\begin{rmk}\label{rmk:unitization-rep}
    Note that if $G$ is nondiscrete we may always extend a representation $\pi \colon L^1 (G,c) \to\mathbb{B}(\mathcal{H})$ to its minimal unitization $L^1 (G,c)^{\sim}$ by forcing the induced representation, also denoted $\pi$, to satisfy $\pi (1_{L^1 (G,c)^\sim}) = \Id_\mathcal{H}$. Indeed we will need to do this in the sequel. If $L^1 (G,c)$ is already unital it will always be implied that $\pi (1_{L^1 (G,c)}) = \Id_{\mathcal{H}}$.
\end{rmk}

\subsection{Symmetric group algebras and $C^*$-uniqueness}
Two concepts that will be of great importance when proving our main result \cref{thm:main-thm} are that of symmetric convolution algebras and $C^*$-uniqueness. We therefore dedicate this section to introducing these concepts along with some important results. In the sequel, if $\A$ is a $*$-algebra and $a \in \A$, we let $\sigma_\A (a)$ denote the spectrum of $a$ in the algebra $\A$.

\begin{defn}\label{defn:symmetric-banach-alg}
A Banach $*$-algebra $\A$ is called \emph{symmetric} if for all $a\in \A$ we have $\sigma_\A (a^*a) \subseteq [0,\infty)$. We will say that a locally compact group $G$ is \emph{symmetric} if $L^1 (G)$ is a symmetric Banach $*$-algebra. 
\end{defn}
Note that this is automatically satisfied in $C^*$-algebras \cite[Theorem 2.2.5]{Murphy90}. By definition of spectrum in a nonunital Banach $*$-algebra we have that nonunital $\A$ is symmetric if its minimal unitization $\Tilde{\A}$ is symmetric. Actually the converse is also true, i.e.\ $\Tilde{A}$ is symmetric if $\A$ is symmetric \cite[Theorem (4.7.9)]{Rickart60}.

Locally compact groups $G$ yielding symmetric (untwisted) convolution algebras $L^1 (G)$ are of importance due to the following result shown in \cite[Theorem 2.8]{grle04} (though noted several times earlier). Note that we can omit the condition that $G$ should be amenable, as it was recently shown that if $L^1 (G)$ is symmetric, then $G$ is amenable \cite[Corollary 4.8]{SameiWiersma20}.
\begin{prop}\label{prop:amenable+symmetric-implies-spec-invariance}
	If $G$ is a locally compact group the following statements are equivalent.
	\begin{enumerate}
		\item [i)] $L^1 (G)$ is symmetric.
		\item [ii)] $\sigma_{L^1 (G)}(f) = \sigma_{\mathbb{B}(L^2 (G))} (L_f)$ for all self-adjoint $f \in L^1 (G)$.
	\end{enumerate}
\end{prop}
Note that for a locally compact group $G$ and a continuous $2$-cocycle $c$ for $G$, the Mackey group $G_c$ is amenable if and only if $G$ is amenable \cite[Proposition 1.13]{Paterson88}.

Like in \cite{grle04}, the proofs of some crucial steps will rely on the following result of Hulanicki, see \cite{Hula72}.

\begin{prop}\label{thm:Hulanicki-result}
	Let $\A$ be a $*$-subalgebra of a Banach $*$-algebra $\B$, and suppose there is a faithful $*$-representation $\pi \colon \B \to \mathbb{B}(\mathcal{H})$, where $\mathcal{H}$ is a Hilbert space. If $\B$ is unital with unit $1_\B$ we require $\pi (1_\B) = \Id_{\mathcal{H}}$. If for all self-adjoint $a \in \A$ we have 
	\begin{equation*}
	\Vert \pi(a) \Vert_{\mathbb{B}(\mathcal{H})} = \lim_{n\to \infty} \Vert a^n \Vert^{1/n}_\B,
	\end{equation*}
	we have 
	\begin{equation*}
	\sigma_\B (a') = \sigma_{\mathbb{B}(\mathcal{H})} (\pi(a'))
	\end{equation*}
	for all self-adjoint $a' \in \A$. 
\end{prop}
For an element $a$ in a Banach $*$-algebra $\A$ the quantity $\lim_{n\to \infty} \Vert a^n \Vert^{1/n}_\A$ is equal to the \emph{spectral radius} $\rho_{\A} (a)$ of $a$ in $\A$ \cite[Theorem 1.2.7]{Murphy90}, and this notation is what we shall mostly use in the sequel.

Locally compact groups yielding symmetric convolution algebras have been studied quite extensively. As examples we mention that all locally compact compactly generated groups of polynomial growth yield symmetric convolution algebras \cite{Losert01}, as do all compact extensions of nilpotent groups \cite[p.\ 191]{Ludwig79}. The latter fact will come into play in \cref{sec:application-Gabor}.

Now let $\pi: G \to \U(\mathcal{H})$ be a faithful $c$-projective unitary representation of $G$ inducing a faithful $*$-representation $\pi : L^1 (G,c) \to \mathbb{B}(\mathcal{H})$. To deduce spectral invariance of $ L^1 (G,c)$ in $\mathbb{B}(\mathcal{H})$ the strategy in \cref{sec:spec-inv} will be to use \cref{thm:Hulanicki-result} to obtain 
\begin{equation*}
\sigma_{L^1 (G,c)} (f) = \sigma_{\mathbb{B}(\mathcal{H})} (\pi(f))
\end{equation*}
for all self-adjoint $f \in L^1 (G,c)$ and then extend to nonself-adjoint elements. It will then be of importance that
\begin{equation*}
	\Vert \tilde{\pi} (j(f)) \Vert = \Vert L_{j(f)} \Vert_{\mathbb{B}(L^2 (G_c))}
\end{equation*} 
for all $f \in L^1 (G,c)$, where $j:L^1 (G,c) \to L^1 (G_c)$ is the isometric $*$-homomorphism from \eqref{eq:j-second-argument-action} and $\tilde{\pi}$ is the faithful $*$-representation of $L^1 (G_c)$ from \eqref{eq:big-representation}. We shall want to consider a class of groups for which this is automatic.
\begin{defn}\label{defn:C*-unique}
	Let $\A$ be a Banach $*$-algebra admitting a faithful $*$-representation. We say $\A$ is \emph{$C^*$-unique} if the maximal $C^*$-norm $\Vert \cdot \Vert_*$ given by
	\begin{equation*}
	\Vert a \Vert_* = \sup \{\Vert \pi (a) \Vert_{\mathbb{B}(\mathcal{H})} \mid \pi: \A \to \mathbb{B}(\mathcal{H}) \text{ is a $*$-representation of $\A$}   \}
	\end{equation*}
	for $a \in \A$, is the unique $C^*$-norm on $\A$. 
	
	We say a locally compact group $G$ is \emph{$C^*$-unique} if $L^1 (G)$ is $C^*$-unique as a Banach $*$-algebra.
\end{defn}
A $C^*$-unique group $G$ is of course amenable, since $C^*$-uniqueness in particular implies $C^*_{\mathrm{red}}(G) = C^* (G)$, i.e.\ that the reduced $C^*$-algebra of $G$ is equal to the full $C^*$-algebra of $G$. The converse is not true \cite{Boidol84, Poguntke93}. There are some known examples of $C^*$-unique groups. As examples we mention semidirect products of abelian groups, connected metabelian groups, as well as groups where every compactly generated subgroup is of polynomial growth \cite[p.\ 224]{Boidol84}. The latter will also come into play in \cref{sec:application-Gabor}.


\section{Spectral invariance of twisted convolution algebras}\label{sec:spec-inv}
All results below will be stated and proved in terms of left representations, i.e.\ left projective unitary representations of groups and left $*$-representations of the twisted convolution algebras we treated in \cref{sec:twisted-conv-algs}. This is only due to left representations being more common in the literature. We note that with proper restatements all results in this section also apply to the case of right representations. Indeed we shall need to consider right representations in \cref{sec:application-Gabor}.

We start by presenting the main theorem of the article. The rest of the section will mostly be dedicated to its proof.

\begin{thm}\label{thm:main-thm}
	Let $G$ be a locally compact group with a continuous $2$-cocycle $c$.
	\begin{enumerate}
	    \item [i)] If $L^1 (G_c)$ is $C^*$-unique, so is $L^1 (G,c)$. 
	    \item [ii)] If $L^1 (G_c)$ is symmetric and $C^*$-unique and $\pi \colon L^1 (G,c)\to \mathbb{B}(\mathcal{H})$ is a faithful $*$-representation, then $f \mapsto \Vert \pi (f)\Vert_{\mathbb{B}(\mathcal{H})}$, $f\in L^1 (G,c)$, is the full $C^*$-norm on $L^1 (G,c)$, and $\sigma_{L^1 (G,c)}(f) = \sigma_{\mathbb{B}(\mathcal{H})}(\pi(f))$ for all $f \in L^1 (G,c)$.
	\end{enumerate}
\end{thm}
\begin{rmk}
    \cref{thm:main-thm} also gives us sufficient conditions for $L^1 (G,c)$ to be symmetric. Namely, from statement ii) in \cref{thm:main-thm} we see that if $L^1(G_c)$ is $C^*$-unique and symmetric, then $L^1 (G,c)$ is spectrally invariant in its (unique) $C^*$-completion. Therefore it is spectrally invariant in its enveloping $C^*$-algebra, which we know happens if and only if $L^1 (G,c)$ (and therefore also its minimal unitization if $G$ is nondiscrete) is symmetric \cite[p.\ 340]{Li12}. 
\end{rmk}

We begin by embedding $L^p (G)$ as a subspace of $L^p (G_c)$ for $1\leq p \leq \infty$. Define the map $j\colon L^p (G)\to L^p (G_c)$ by
\begin{equation}\label{eq:j-second-argument-action}
	j(f)(x,\tau) = \tau f(x).
\end{equation}
\begin{lemma}\label{lemma:j-*-hom}
    Let $G$ be a locally compact group and let $c$ be a continuous $2$-cocycle for $G$. Then
	$j$ defined by \eqref{eq:j-second-argument-action} is an isometric $*$-homomorphism from $L^1 (G,c)$ to $L^1 (G_{c})$, and an isometry from $L^p (G)$ to $L^p (G_{c})$ for $1<p\leq \infty$. Moreover, if $f \in L^1 (G,c)$ and $g \in L^p (G)$, we have
	\begin{equation}\label{eq:j-preserves-convolution-action}
	j( f \natural_c g ) = j(f) * j(g)
	\end{equation}
	for $p \in [1,\infty]$.
\end{lemma}
\begin{proof}
	We begin by verifying that $j$ is an isometry for $1\leq p < \infty$. Let $f \in L^p (G)$. Then
	\begin{equation*}
	\begin{split}
	\Vert j(f) \Vert^p_{L^p (G_c)} &= \int_{G_{c}} \vert j(f) ((x,\tau))\vert^p \dif \tau \dif x 
	= \int_G \int_\T \vert \tau f(x) \vert^p \dif \tau \dif x \\
	&= \int_G \vert f(x) \vert^p \dif x = \Vert f \Vert^p_{L^p (G)}.
	\end{split}
	\end{equation*}
	Likewise, for $p = \infty$ and $f \in L^\infty (G)$ we get
	\begin{equation*}
	\Vert j(f) \Vert_{L^\infty (G_{c})} = \sup_{(x,\tau)\in G_{c}} \vert j(f)((x,\tau))\vert = \sup_{(x,\tau)\in G_{c} } \vert \tau f(x)\vert = \sup_{x\in G}\vert f(x)\vert = \Vert f\Vert_{L^\infty (G)}.
	\end{equation*}
	We now verify that $j$ is a $*$-homomorphism when $p=1$. Let $f_1 , f_2 \in L^1 (G,c)$. Then for all $(x,\tau)\in G_c$ we have
	\begin{equation*}
	\begin{split}
	(j(f_1) * j(f_2))((x,\tau)) &= \int_{G_{c}} j(f_1)((y,\xi)) j(f_2)((y,\xi)^{-1} (x,\tau)) \dif \xi \dif y \\
	&= \int_G \int_\T j(f_1)((y,\xi)) j(f_2) ((y^{-1} x, \overline{\xi} c(y,y^{-1}) \tau \overline{c(y^{-1}, x)})) \dif \xi \dif y \\
	&= \int_G \int_\T \xi f_1 (y) \overline{\xi} \tau c(y,y^{-1}) \overline{c(y^{-1}, x)} f_2 (y^{-1}x) \dif \xi \dif y \\
	&= \tau \int_G f_1 (y) f_2 (y^{-1} x) c(y,y^{-1}) \overline{c(y^{-1},x)} \dif y \\
	&= \tau\int_G f_1 (y) f_2 (y^{-1} x) c(y,y^{-1}x) \dif y \\
	&= j(f_1 \natural_{c} f_2) ((x,\tau)),
	\end{split}
	\end{equation*}
	where we in the second to last line used \eqref{eq:2-cocycle-reduction}. Doing the same calculation with $f_2 \in L^p (G)$ shows that \eqref{eq:j-preserves-convolution-action} holds.
	
	It then remains to show that $j$ respects the involutions. For $f \in L^1 (G,c)$ and all $(x,\tau) \in G_{c}$, we have 
	\begin{equation*}
	\begin{split}
	j(f)^* ((x, \tau)) &= m(x^{-1}) \overline{j(f) ((x,\tau)^{-1} )} = m(x^{-1}) \overline{j(f) ((x^{-1}, \overline{\tau} c(x,x^{-1})  )}\\
	&=  m(x^{-1}) \overline{\overline{\tau} c(x,x^{-1})} \overline{f(x^{-1})}  
	= m(x^{-1})\tau \overline{c(x^{-1},x)f(x^{-1})}\\
	&= \tau f^{*_c} (x) = j(f^{*_c})((x,\tau)) .
	\end{split}
	\end{equation*}
	Hence $j(f)^* = j(f^{*_c})$ for all $f \in L^1 (G,c)$. This finishes the proof.
\end{proof}

Since $j$ is an isometry and $L^p (G)$ is complete for all $p \in [1 , \infty]$, we get that $j(L^p (G))$ is a closed subspace of $L^p (G_{c})$. We may actually obtain a quite explicit description of this subspace. To do this, we expand functions in $L^p (G_{c})$ as Fourier series with respect to their second argument, that is, in the $\T$-variable. Since the measure on $G_{c}$ is the product measure coming from $G$ and $\T$, we have that for any $F \in L^p (G_{c})$, $1\leq p \leq \infty$, and any $x\in G$, the function $\tau \mapsto F(x,\tau)$ is in $L^p (\T) \subseteq L^1 (\T)$. Therefore the Fourier coefficients
\begin{equation}\label{eq:Fourier-coeffs}
F_k (x) =\int_\T F(x,\tau) \overline{\tau}^k \dif \tau
\end{equation}
are well-defined, and the resulting Fourier series
\begin{equation*}
F(x,\tau) = \sum_{k\in \Z} F_k (x) \tau^k
\end{equation*}
converges in $L^p (\T)$ for $1<p<\infty$. The following lemma then describes the range of $j$.

\begin{lemma}\label{lemma:range-of-j}
    Let $G$ be a locally compact group and let $c$ be a continuous $2$-cocycle for $G$. For $1\leq p \leq \infty$ we have $j(L^p (G)) = \{ F \in L^p (G_{c}) \mid F_k = 0 \text{ for $k\neq 1$}\}$.
\end{lemma}
\begin{proof}
	The inclusion $j(L^p (G)) \subseteq \{ F \in L^p (G_{c}) \mid F_k = 0 \quad \text{for $k\neq 1$}\}$ is immediate by \eqref{eq:j-second-argument-action} and \eqref{eq:Fourier-coeffs}. For the converse containment note that if $F \in \{ F \in L^p (G_{c}) \mid F_k = 0  \text{ for $k\neq 1$}\}$, then for all $(x,\tau)\in G_{c}$ we have $F(x,\tau) = \tau F_{1} (x)$. Since the measure on $G_{c}$ is the product measure we must have that $x \mapsto F_{1}(x)$ is in $L^p (G)$. Hence $F = j(F_{1})$, which proves the lemma.  
\end{proof}

To simplify notation in the sequel, denote by $L^1 (G_c)_n$ the set
\begin{equation*}\label{eq:defn-ksubspace}
	L^1 (G_c)_n := \{F \in L^1 (G_c) \mid F_k = 0  \text{ for $k\neq n$}   \}.
\end{equation*}
It is then immediate that $L^1 (G_c)_1 = j(L^1 (G,c))$. We also have the following result.
\begin{prop}\label{prop:convolution-and-involution-identities}
	Let $G$ be a locally compact group with a continuous $2$-cocycle $c$, let $F \in L^1 (G_c)$ and let $H\in L^p (G_c)$ for some $1 \leq p < \infty$. Then 
	\begin{equation}\label{eq:convolution-identity}
	( F * H)((x,\tau)) = \sum_{n \in \Z} (F_n \natural_{c^n} H_n)(x) \tau^n,
	\end{equation}
	for all $(x,\tau)\in G_c$, where $c^n$ is $c$ to the nth power as in \cref{lemma:useful-cocycle-props}. Moreover,
	\begin{equation}\label{eq:involution-identity}
		(F_n)^{*_{c^n}} = (F^*)_n
	\end{equation}
	for all $n \in \Z$.
\end{prop}
\begin{proof}
	Below we will make use of the Fourier expansions $F(y,\xi) = \sum_{m \in \Z} F_m (y) \xi^m$ and $H(y,\xi) = \sum_{m\in \Z} H_m (y) \xi^m$, where $F_m$ and $H_m$ are obtained through \eqref{eq:Fourier-coeffs}. We will assume both $F$ and $H$ have finite expansions of the form \eqref{eq:Fourier-coeffs}. This is sufficient since trigonometric polynomials are dense in $L^p (\T)$, $1\leq p < \infty$. The extension to the full statement follows by a standard density argument. 
	
	Since $\{\xi^m\}_{m\in \Z}$ is an orthonormal system in $L^2 (\T)$, we have for all $(x,\tau)\in G_{c}$
	\begin{equation*}
	\begin{split}
	(F * H)((x,\tau)) &= \int_G \int_\T F((y,\xi)) H ((y,\xi)^{-1} (x,\tau)) \dif \xi \dif y \\
	&= \int_G \int_\T F((y,\xi)) H((y^{-1}x, \overline{\xi}  c(y^{-1},y) \tau \overline{c(y^{-1},x)}) \dif \xi \dif y \\
	&= \int_G \int_\T \sum_{m\in \Z} F_m (y) \xi^m \cdot \sum_{n\in \Z} H_n (y^{-1}x) \overline{\xi}^n \tau^n (c(y,y^{-1}x))^n \dif \xi \dif y \\
	&= \int_G \sum_{n\in \Z} F_n (y) H_n (y^{-1}x) c^n (y,y^{-1}x) \tau^n \dif y\\
	&= \sum_{n\in \Z} \bigg( \int_G F_n (y) H_n (y^{-1}x) c^n (y,y^{-1}x) \dif y \bigg) \tau^n \\
	&= \sum_{n \in \Z} (F_n \natural_{c^n} H_n)(x) \tau^n
	\end{split}
	\end{equation*}
	where we at the third equality used \eqref{eq:2-cocycle-reduction}. This establishes \eqref{eq:convolution-identity}.
	
	For any $F \in L^1 (G_c)$ we also have
	\begin{equation*}
		\begin{split}
		(F^*)_n (x) &= \int_{\T} F^* ((x,\tau)) \overline{\tau}^n \dif \tau
		= \int_{\T} m(x^{-1})\overline{F((x^{-1}, \overline{\tau} c(x^{-1},x)))}  \overline{\tau}^n \dif \tau \\
		&= m(x^{-1}) \int_\T \overline{F((x^{-1} , \tau c(x^{-1},x)))} \tau^n \dif \tau 
		= m(x^{-1})\int_\T \overline{F((x^{-1}, \tau))} \tau^n \overline{c(x^{-1},x)^n} \dif \tau \\
		&= m(x^{-1})\overline{c(x^{-1},x)^n} \overline{\int_{\T} F((x^{-1} , \tau)) \overline{\tau}^n \dif \tau} = m(x^{-1})\overline{c^n(x^{-1},x)} \overline{F_n (x^{-1})}\\
		&= (F_n)^{*_{c^n}} (x),
		\end{split}
	\end{equation*}
	for all $x \in G$, which establishes \eqref{eq:involution-identity}.
\end{proof}
The following corollary is then immediate.
\begin{corl}\label{corl:graded-subspace-iso}
    Let $G$ be a locally compact group and let $c$ be a continuous $2$-cocycle for $G$. Then
	$L^1 (G_c)_n \cong L^1 (G,c^n)$ as Banach $*$-algebras. 
\end{corl}

As a final preparation before proving \cref{thm:main-thm}, we need the following lemma.

\begin{lemma}\label{lemma:spectral-radius-for-1-2}
    Let $G$ be a locally compact group and let $c$ be a continuous $2$-cocycle for $G$. 
	For $f \in L^1 (G,c)$ we then have
	\begin{equation*}
	\rho_{L^1 (G,c)}(f) = \rho_{L^1 (G_{c})} (j(f)).
	\end{equation*}
	If in addition $f$ is self-adjoint we get
	\begin{equation}\label{eq:twisted-left-reg-faithful}
	\rho_{\mathbb{B}(L^2(G))}(L^c_f) = \rho_{\mathbb{B}(L^2 (G_{c}))} (L_{j(f)}).
	\end{equation}
\end{lemma}
\begin{proof}
	Since $j \colon L^1 (G,c) \to L^1 (G_{c})$ is an isometric $*$-homomorphism we have
	\begin{equation*}
	\rho_{L^1 (G,c)}(f) = \lim_{n\to \infty} \Vert f^n \Vert^{1/n}_{L^1 (G,c)} = \lim_{n\to \infty} \Vert j(f)^n \Vert^{1/n}_{L^1(G_{c})}= \rho_{L^1 (G_{c})} (j(f)),
	\end{equation*}
	which proves the first statement. 
	
	For the second statement, let $f\in L^1 (G,c)$ be self-adjoint. Since $f$ is self-adjoint and $L^c_f$ and $L_{j(f)}$ realize $f$ and $j(f)$ as bounded operators on Hilbert spaces, i.e.\ as elements of a $C^*$-algebra, we have
	\begin{equation}
	\rho_{\mathbb{B}(L^2(G))}(L^c_f) = \Vert L^c_f \Vert_{\mathbb{B}(L^2(G))} \quad \text{and} \quad \rho_{\mathbb{B}(L^2 (G_{c}))} (L_{j(f)}) = \Vert L_{j(f)}\Vert_{\mathbb{B}(L^2 (G_{c}))},
	\end{equation}
	see \cite[Theorem 2.1.1]{Murphy90}. Hence it suffices to show that $\Vert L^c_f \Vert_{\mathbb{B}(L^2(G))} = \Vert L_{j(f)}\Vert_{\mathbb{B}(L^2 (G_{c}))}$. To do this, note first that by \cref{lemma:j-*-hom}
	\begin{equation*}
	L_{j(f)} j(g) = j(f) * j(g) = j(f \natural_c g) = j(L^c_f  g)
	\end{equation*}
	for any $g \in L^2 (G)$. Moreover, by \cref{prop:convolution-and-involution-identities} we see that $L_{j(f)}\vert_{j(L^2 (G))^{\perp}}=0$. Since $j: L^2 (G) \to L^2 (G_{c})$ is an isometry it then follows that $\Vert L^c_f \Vert_{\mathbb{B}(L^2(G))} = \Vert L_{j(f)}\Vert_{\mathbb{B}(L^2 (G_{c}))}$, which finishes the proof.
\end{proof}

We are finally ready to prove \cref{thm:main-thm}.
\begin{proof}[Proof of \cref{thm:main-thm}]
We begin by proving i). Let $\pi \colon L^1 (G,c)\to \mathbb{B}(\mathcal{H})$ be a faithful $*$-representation. As $G_c$ is assumed $C^*$-unique, $G_c$ is in particular amenable, so it follows that $G$ is also amenable. Then \cref{prop:Leptin-twisted-result} gives that $f \mapsto \Vert L^c_f\Vert_{\mathbb{B}(L^2(G))}$, $f\in L^1 (G,c)$, is the maximal $C^*$-norm on $L^1 (G,c)$. Hence it suffices to prove that $\Vert \pi (f)\Vert_{\mathbb{B}(\mathcal{H})} = \Vert L^c_f\Vert_{\mathbb{B}(L^2 (G))}$ for all $f \in L^1 (G,c)$. To do this, we will first extend $\pi$ to a faithful $*$-representation of $L^1 (G_c)$. The obvious attempt at a $*$-representation of $L^1 (G_c)$, namely the integrated representation of $\pi_c \colon G_c \to \U(\mathcal{H})$ as in \eqref{eq:induced-rep-on-Mackey}, is in general not faithful as noted at the end of \cref{subsec:obtaining-C*-algs}. The construction of the desired faithful $*$-representation $\Tilde{\pi}$ of $L^1 (G_c)$ is therefore more involved. 
	
For all $n \in \Z$ we know by \cref{corl:graded-subspace-iso} that $L^1(G_c)_n \cong L^1 (G,c^n)$ as Banach $*$-algebras, and in the sequel we make this identification to ease notation. For any $n \in \Z \setminus \{1\}$ we define 
	\begin{equation*}
		\pi^{(n)} := L^{c^n} \colon L^1 (G,c^n) \to \mathbb{B}(L^2 (G)),
	\end{equation*}
	and set 
	\begin{equation*}
		\pi^{(1)} := \pi \colon L^1 (G,c)\to \mathbb{B}(\mathcal{H}).
	\end{equation*}
	Then $\pi^{(n)}$ is a faithful $*$-representation of $L^1 (G,c^n)$ for all $n \in \Z$. 
	Moreover, we set
	\begin{equation*}\label{eq:large-hilbert-space}
	\mathcal{H}^{(n)} =
	\begin{cases*}
	L^2 (G) & if $n\in \Z\setminus \{1\}$ \\
	\mathcal{H}        & if $n=1$.
	\end{cases*}
	\end{equation*}
	We then consider the map $\Tilde{\pi}: L^1 (G_c) \to \oplus_{k\in \Z} \mathbb{B}(\mathcal{H}^{(k)})$ which for $F \in L^1 (G_c)$ is given by
	\begin{equation}\label{eq:big-representation}
	F \mapsto (F_k)_{k\in \Z} \mapsto \bigoplus_{k\in \Z} \pi^{(k)}(F_k).
	\end{equation}
	We must verify that this is a faithful $*$-homomorphism. Continuity will then follow since any $*$-homomorphism from a Banach $*$-algebra to a $C^*$-algebra is continuous \cite[Theorem 2.1.7]{Murphy90}
	
	For $F,H \in L^1 (G_c)$ it then follows from \eqref{eq:convolution-identity} that
	\begin{equation*}
		\begin{split}
		\Tilde{\pi}(F* H) = \bigoplus_{k\in \Z} \pi^{(k)} (F_k \natural_{c^k} H_k) = \bigoplus_{k\in \Z} \pi^{(k)}(F_k) \circ \pi^{(k)}(H_k) = \Tilde{\pi}(F)\Tilde{\pi}(H).
		\end{split}
	\end{equation*}
	It also follows from \eqref{eq:involution-identity} that
	\begin{equation*}
		\Tilde{\pi}(F^*) = \bigoplus_{k\in \Z}  \pi^{(k)} ((F^*)_k) = \bigoplus_{k\in \Z} \pi^{(k)} ((F_k)^{*_{c^k}}) = \bigoplus_{k\in \Z}  \pi^{(k)} (F_k)^* = \Tilde{\pi}(F)^*.
	\end{equation*}
	We conclude that $\Tilde{\pi}$ is a continuous $*$-homomorphism.
	
	Now suppose $F \in L^1 (G_c)$ is such that $\Tilde{\pi}(F) = 0$. Then $\pi^{(k)}(F_k) = 0$ for all $k\in \Z$, and since $\pi^{(k)} \colon L^1 (G,c^k) \to \mathbb{B}(\mathcal{H}^{(k)})$ are all faithful, we conclude that $F_k =0$ for all $k \in \Z$. Since the Fourier transform is injective on $L^1$, this happens if and only if $F=0$ almost everywhere, i.e.\ if $F=0$ in $L^1 (G_c)$. We deduce that $\Tilde{\pi}$ is a faithful $*$-homomorphism. 
	
	Observe that since $\mathcal{H}^{(1)} = \mathcal{H}$, the two representations $\pi \colon L^1 (G,c) \to \mathbb{B}(\mathcal{H})$ and $\Tilde{\pi} \circ j \colon L^1 (G,c) \to \mathbb{B}(\mathcal{H}^{(1)})$ can naturally be identified. Using the $C^*$-identity and \cref{lemma:spectral-radius-for-1-2} we then obtain
	\begin{equation*}
	    \begin{split}
	        \Vert \pi (f) \Vert^2_{\mathbb{B}(\mathcal{H})} &= \Vert \pi (f^* \natural_c f) \Vert_{\mathbb{B}(\mathcal{H})} = \Vert \tilde{\pi}(j(f^* \natural_c f)) \Vert_{\oplus_{k\in \Z}\mathbb{B}(\mathcal{H}^{(k)})} \\
	        &= \Vert L_{j(f^* \natural_c f)}\Vert_{\mathbb{B}(L^2 (G_c))}
	        = \Vert L^c_{f^* \natural_c f}\Vert_{\mathbb{B}(L^2(G))} = \Vert L^c_f \Vert^2_{\mathbb{B}(L^2 (G))}
	    \end{split}
	\end{equation*}
	for all $f \in L^1 (G,c)$, which proves i). 
	
	To prove ii), let first $f\in L^1 (G,c)$ be self-adjoint. Using \cref{lemma:spectral-radius-for-1-2}, \cref{prop:amenable+symmetric-implies-spec-invariance} and i) of \cref{thm:main-thm}, we have the following chain of equalities
	\begin{equation*}
	    \begin{split}
	        \rho_{L^1 (G,c)}(f) &= \rho_{L^1 (G_c)}(j(f)) = \rho_{\mathbb{B}(L^2 (G_c))}(L_{j(f)}) \\
	        &= \rho_{\mathbb{B}(L^2 (G))}(L^c_f) = \Vert L^c_f \Vert_{\mathbb{B}(L^2 (G))} = \Vert \pi (f) \Vert_{\mathbb{B}(\mathcal{H})}.
	    \end{split}
	\end{equation*}
	By \cref{thm:Hulanicki-result} it then follows that $\sigma_{L^1 (G,c)}(f) = \sigma_{\mathbb{B}(\mathcal{H})}(\pi (f))$ for all self-adjoint $f \in L^1 (G,c)$. 
	
	In the following we will assume $L^1 (G,c)$ is unital. If this is not the case we may do the same argument by going to the minimal unitization $L^1 (G,c)^\sim$ and lift the $*$-representation as in \cref{rmk:unitization-rep}.
	Now let $f \in L^1 (G,c)$ be an arbitrary element for which $\pi (f)$ is invertible in $\mathbb{B}(\mathcal{H})$. Then $f \natural_c f^*$ is self-adjoint and $\pi (f \natural_c f^*)$ is also invertible in $\mathbb{B}(\mathcal{H})$. By the above, $f \natural_c f^*$ is invertible in $L^1 (G,c)$, and likewise we obtain that $f^* \natural_c f$ is invertible in $L^1 (G,c)$. Then $f^* \natural_c (f \natural_c f^*)^{-1}$ is a right inverse of $f$ in $L^1 (G,c)$, and $(f^* \natural_c f)^{-1} \natural_c f^*$ is a left inverse of $f$ in $L^1 (G,c)$, which implies that $f$ is invertible in $L^1 (G,c)$. Applying the same argument to elements of the form $f - \lambda \cdot 1$, $\lambda \in \C$, we get that
	\begin{equation*}
	\sigma_{L^1 (G,c)} (f) \subseteq \sigma_{\mathbb{B}(\mathcal{H})}(\pi (f)).
	\end{equation*}
	Since $f \mapsto \pi (f)$ is a $*$-representation the converse containment always holds. Hence we conlude that $\sigma_{L^1 (G,c)} (f)= \sigma_{\mathbb{B}(\mathcal{H})}(\pi (f)) $ for all $f \in L^1 (G,c)$.
\end{proof}

\section{Applications to Gabor analysis}\label{sec:application-Gabor}
We begin by introducing the central concepts of Gabor analysis, before formulating the main result of this section. We then rephrase the setting of the problem in terms spectral invariance of a certain convolution algebra and use \cref{thm:main-thm} to prove the result. 

Throughout this section $G$ will be a locally compact abelian group and $\widehat{G}$ will be its Pontryagin dual. Note that we will write the group operation additively. 
Moreover, $\Sub$ will denote a closed cocompact subgroup of the time-frequency plane $\tfp{G}$. The reason for restricting to cocompact subgroups will be made clear in \cref{rmk:cocompact-necessary}. We fix a Haar measure on $G$ and equip $\widehat{G}$ with the dual measure such that Plancherel's formula holds \cite[Theorem 3.4.8]{DeEc14}. We also fix a Haar measure on $\Sub$, and give $(\tfp{G})/\Sub$ the unique measure such that Weil's formula holds \cite[equation (2.4)]{JaLe16}. The \emph{size of $\Sub$} is the quantity $s(\Sub) := \mu ((\tfp{G})/\Sub)$, where $\mu$ is the chosen Haar measure. As $\Sub$ is cocompact in $\tfp{G}$, we have $s(\Sub) < \infty$.

We proceed to introduce the two unitary operators most relevant for Gabor analysis. Given $x \in G$ and $\omega \in \widehat{G}$, we define the \emph{translation operator} $T_x$ and \emph{modulation operator} $M_\omega$ on $L^2 (G)$ by
\begin{equation*}
	(T_x f)(t) = f(t-x), \quad (M_\omega f)(t) = \omega (t) f(t)
\end{equation*}
for $f \in L^2 (G)$ and $t \in G$. Moreover, we define a \emph{time-frequency shift} by
\begin{equation}\label{eq:tf-shift}
	\pi (x,\omega) := M_\omega T_x
\end{equation}
for $x \in G$ and $\omega \in \widehat{G}$.

Having introduced both translation and modulation we may define the subgroup of $\tfp{G}$ which will be of greatest importance to us when proving \cref{thm:continuous-frame-inverse-regularity}. This is due to the reformulations of the frame operator in \eqref{eq:frame-op-rewritten} and \eqref{eq:frame-op-module-reform} below. The \emph{adjoint subgroup} of $\Sub$, denoted $\Subo$, is the closed subgroup of $\tfp{G}$ defined by
\begin{equation}\label{eq:adjoint-subgroup}
	\Subo := \{w \in \tfp{G} \mid \pi (z)\pi(w) = \pi(w)\pi(z) \text{ for all $z \in \Sub$} \}.
\end{equation}
We may identify $\Subo$ with $((\tfp{G})/\Sub)^{\widehat{}}$ as in \cite[p. 234]{JaLe16} to pick the dual measure on $\Subo$ corresponding to the measure on $(\tfp{G})/\Sub$. $\Sub$ is cocompact in $\tfp{G}$, so $\Subo$ 
is discrete. The induced measure on $\Subo$ is the counting measure scaled with the constant $s(\Sub)^{-1}$ \cite[equation (13)]{jalu18duality}.

Given $g \in L^2 (G)$, a \emph{Gabor system over $\Sub$ with generator $g$} is a family $\GS (g;\Sub) := (\pi (z) g)_{z\in \Sub}$. It is called a \emph{Gabor frame} if it is a (continuous) frame for $L^2 (G)$ \cite{AlAn93, JaLe16, Ka94} in the sense that the following conditions are satisfied:
\begin{enumerate}
	\item [i)] The family $\GS (g,\Sub)$ is weakly measurable, i.e.\ for every $f \in L^2 (G)$ the map $z \mapsto \hs{f}{\pi (z)g}$ is measurable.
	\item [ii)] There exist positive constants $C,D >0$ such that for all $f \in L^2 (G)$ we have that
	\begin{equation*}
		C \Vert f \Vert_2^2 \leq \int_{\Sub} \vert \hs{f}{\pi (z) g} \vert^2 \dif z \leq D \Vert f \Vert_2^2.
	\end{equation*}
\end{enumerate}
\begin{rmk}\label{rmk:cocompact-necessary}
	Gabor frames $\GS(g;\Sub)$ for $L^2 (G)$ with $g \in L^2 (G)$ can only exist if $\Sub$ is cocompact \cite[Theorem 5.1]{JaLe16}. Indeed, this is also the case if we consider finitely many functions $g_1, \ldots ,g_k \in L^2 (G)$ and a Gabor system $\GS(g_1, \ldots , g_k ; \Sub)$ as in \cref{rmk:multiwindow-extension} below \cite[Lemma 4.9]{jalu18duality}. The same is true if we consider matrix frames introduced in \cite{AuJaLu19}, see \cite[Proposition 4.34]{AuJaLu19}.
\end{rmk}
If $\GS(g;\Sub)$ is weakly measurable and $D < \infty$ for this family, we say $\GS(g;\Sub)$ is a \emph{Bessel system}. Associated to any Bessel system $\GS(g;\Sub)$ is a linear bounded operator known as the \emph{frame operator associated to} $\GS(g;\Sub)$. It is the operator
\begin{equation*}
	\begin{split}
	S \colon L^2 (G) &\to L^2 (G) \\
	f &\mapsto \int_{\Sub} \hs{f}{\pi (z) g} \pi (z) g \dif z,
	\end{split}
\end{equation*}
where we interpret the integral weakly. 
It is well-known in frame theory that $S$ commutes with all time-frequency shifts $\pi (z)$ when $z \in \Sub$, and that $\GS(g;\Sub)$ is a Gabor frame for $L^2 (G)$ if and only if $S$ is invertible on $L^2 (G)$. Moreover, it is not hard to see that $S$ is a positive operator. 

Now let $\GS(g;\Sub)$ be a Gabor frame for $L^2 (G)$. We have
\begin{equation}\label{eq:dual-atom-reconstruction}
	f = S^{-1} S f = \int_{\Sub} \hs{f}{\pi(z)g} \pi (z) S^{-1}g \dif z
\end{equation}
for all $f \in L^2 (G)$. The function $S^{-1}g$ is known as \emph{the canonical dual atom of $g$}. Moreover, we have
\begin{equation}\label{eq:tight-atom-reconstruction}
	f = S^{-1/2}S S^{-1/2} f = \int_{\Sub} \hs{f}{\pi(z) S^{-1/2}g} \pi(z) S^{-1/2} g \dif z
\end{equation}
for all $f \in L^2 (G)$. The function $S^{-1/2}g$ is known as \emph{the canonical tight atom associated to $g$}.

As a last preparation before presenting the main result of this section we must introduce a function space. Let $g \in L^2 (G)$. We define the \emph{short-time Fourier transform with respect to} $g$ to be the operator $V_g : L^2 (G) \to L^2 (\tfp{G})$ given by
\begin{equation*}
V_g f (z) = \hs{f}{\pi (z) g}
\end{equation*} 
for $f \in L^2 (G)$ and $z \in \tfp{G}$. Using this, we define the \emph{Feichtinger algebra} $S_0 (G)$ by
\begin{equation}\label{eq:def-feichtinger-alg}
	S_0 (G) := \{f \in L^2 (G) \mid V_f f \in L^1 (\tfp{G})   \}.
\end{equation}
The Feichtinger algebra is known as a nice space of test functions for time-frequency analysis, and its elements have good decay in both time and frequency. We refer the reader to \cite{ja18} for more information on the Feichtinger algebra. At last, we may state the main theorem of this section.

\begin{thm}\label{thm:continuous-frame-inverse-regularity}
	Let $\Sub \subseteq \tfp{G}$ be a closed cocompact subgroup, and suppose $g \in S_0 (G)$ is such that $\GS(g;\Sub)$ is a Gabor frame for $L^2 (G)$. Then $S^{-1}g, S^{-1/2}g \in S_0 (G)$ as well. 
\end{thm}
In the case when $\Sub$ is a separable lattice in $\tfp{\R^d}$, \cref{thm:continuous-frame-inverse-regularity} was proved in \cite{grle04}, and it was claimed to hold for general lattices in phase spaces of locally compact abelian groups. It is possible that their techniques can be adapted to the setting of closed cocompact subgroups. However, it will turn out that the result is easier to deduce by using \cref{thm:main-thm}, thereby circumventing any need to use the periodization techniques of \cite{grle04}. 
In order to show \cref{thm:continuous-frame-inverse-regularity} we will reformulate the above setup to incorporate twisted convolution algebras. As a first step towards this we present the Fundamental Identity of Gabor Analysis. We refer the reader to \cite[Proposition 2.11]{ri88} for a proof. There the Schwartz-Bruhat space is used, but the proof can easily be adapted to the case of $S_0 (G)$.
\begin{prop}\label{prop:FIGA}
	Let $f,g,h \in S_0 (G)$. Then
	\begin{equation*}\label{eq:FIGA}
		\int_{\Sub} \hs{f}{\pi (z)g} \pi (z) h \dif z = \frac{1}{s(\Sub)}\sum_{w\in \Subo} \hs{\pi (w) h}{g} \pi(w)^* f
	\end{equation*}
	where we interpret the integral and the sum weakly.
\end{prop}
\cref{prop:FIGA} allows us to rewrite the frame operator $S$ for $\GS(g;\Sub)$ as
\begin{equation}\label{eq:frame-op-rewritten}
	S f = \int_{\Sub} \hs{f}{\pi(z)g} \pi (z)g \dif z = \frac{1}{s(\Sub)} \sum_{w\in \Subo} \hs{\pi(w)g}{g} \pi(w)^* f.
\end{equation}
This observation is key in rephrasing the problem. It is the right hand side which is of importance to us, and it will be most natural to restate the frame operator in terms of a right $*$-representation of a twisted convolution algebra, see \cref{eq:frame-op-module-reform}. 

We will also need the continuous $2$-cocycle on $\tfp{G}$ known as the \emph{Heisenberg $2$-cocycle} \cite[p.\ 263]{ri88}. It is the map $c: (\tfp{G}) \times (\tfp{G}) \to \T$ given by
\begin{equation}\label{eq:heisenberg-cocycle}
	c((x,\omega), (y,\tau)) = \overline{\tau(x)}
\end{equation}
for $(x,\omega), (y,\tau) \in \tfp{G}$. Restricting to $\Subo$, we construct the convolution algebra $\ell^1 (\Subo, \cbar)$ as in \cref{sec:twisted-conv-algs}. Now the map
\begin{equation*}
	\begin{split}
	\pi^* \colon \tfp{G} &\to \U(L^2 (G)) \\
	(x,\omega) &\mapsto \pi(x,\omega)^*
	\end{split}
\end{equation*}
defines a right $\cbar$-projective unitary representation. Restricting to $\Subo$ we likewise get a right $\cbar$-projective unitary representation of $\Subo$, which we also denote by $\pi^*$. The integrated representation defines a right $*$-representation $\pi^* \colon \ell^1 (\Subo,\cbar)\to \mathbb{B}(L^2 (G))$. This representation leaves $S_0 (G)$ invariant, i.e.\ $\pi^* (\ell^1 (\Subo,\cbar))S_0 (G) \subseteq S_0 (G)$ \cite[Theorem 3.4]{jalu18duality}. Given $a= (a_w)_{w\in \Subo}\in \ell^1 (\Subo,\cbar)$ and $f \in L^2 (G)$ we have
\begin{equation*}
	\pi^* (a) f = \frac{1}{s(\Sub)}\sum_{w\in \Subo} a_w \pi(w)^* f.
\end{equation*}
Also, this $*$-representation is known to be faithful \cite[Proposition 2.2]{ri88}. Moreover, for $g \in S_0 (G)$ we have $(\hs{\pi(w)g}{g})_{w\in \Subo} \in \ell^1 (\Subo, \cbar)$ \cite[Theorem 3.4]{jalu18duality}. Using \eqref{eq:frame-op-rewritten} for the Gabor system $\GS(g;\Sub)$, $g \in S_0 (G)$, we now see that
\begin{equation}\label{eq:frame-op-module-reform}
	Sf = \pi^* ((\hs{\pi (w)g}{g})_{w\in \Subo})f
\end{equation}
for $f \in L^2 (G)$. We are finally ready to prove \cref{thm:continuous-frame-inverse-regularity}.

\begin{proof}[Proof of \cref{thm:continuous-frame-inverse-regularity}]
	If $g \in S_0 (G)$ is such that $\GS(g;\Sub)$ is a Gabor frame for $L^2 (G)$, then the corresponding frame operator $S$ is invertible. By \eqref{eq:frame-op-module-reform} we may write $Sf = \pi^* ((\hs{\pi (w)g}{g})_{w\in \Subo})f$ for any $f \in L^2 (G)$. In other words, if $S$ is invertible as an operator on $L^2 (G)$, then $\pi^* ((\hs{\pi (w)g}{g})_{w\in \Subo})$ is also invertible in $\mathbb{B} (L^2 (G))$. Since $\Subo$ is abelian, every compactly generated subgroup of $\Subo$ is of polynomial growth by the structure theorem for compactly generated locally compact abelian groups \cite[Theorem 4.2.2]{DeEc14}. Hence every compactly generated subgroup of $\Subo_{\cbar}$ is also of polynomial growth, since it is a compact extension of $\Subo$. It follows that $\Subo_{\cbar}$ is $C^*$-unique. Moreover, $\Subo_{\cbar}$ is nilpotent of class $1$ as $\Subo$ is abelian, so it follows that $\ell^1 (\Subo_{\cbar})$ is symmetric. By \cref{thm:main-thm} we then have that $\ell^1 (\Subo,\cbar)$ is spectrally invariant in $\mathbb{B}(L^2 (G))$. Hence there is $a = (a_w)_{w\in \Subo}\in \ell^1 (\Subo,\cbar)$ such that $a \natural_{\cbar} (\hs{\pi (w)g}{g})_{w\in \Subo} = 1_{\ell^1 (\Subo,\cbar)} = (\hs{\pi (w)g}{g})_{w\in \Subo} \natural_{\cbar} a$ and 
	\begin{equation*}
		S^{-1}f = \pi^* (a)f
	\end{equation*}
	for all $f \in L^2 (G)$. Since $\pi^* (\ell^1 (\Subo,\cbar))$ leaves $S_0 (G)$ invariant, it follows that $S^{-1}g \in S_0 (G)$.
	
	Since $S$, hence also $S^{-1}$, is a positive operator, we may also take the square root of the image of $a$ under $\pi^*$ in $\mathbb{B}(L^2 (G))$. By spectral invariance and the fact that Banach $*$-algebras are closed under holomorphic functional calculus \cite[p.\ 212]{Dales2000} it follows that there is $b= (b_w)_{w\in \Subo} \in \ell^1 (\Subo,\cbar)$ such that
	\begin{equation*}
		S^{-1/2}f = \pi^* (b)f
	\end{equation*} 
	for all $f \in L^2 (G)$. Once again, since $\pi^* (\ell^1 (\Subo,\cbar))$ leaves $S_0 (G)$ invariant, it follows that $S^{-1/2}g \in S_0 (G)$. This finishes the proof.
\end{proof}

\begin{rmk}\label{rmk:multiwindow-extension}
	There are no issues extending this to multi-window Gabor frames, i.e.\ the case of $g_1, \ldots, g_k \in S_0 (G)$ such that $\GS(g_1,\ldots, g_k ;\Sub) := \GS(g_1 ;\Sub) \cup \cdots \cup \GS(g_k ; \Sub)$ is a frame for $L^2 (G)$. Indeed, the only real difference is that we in \eqref{eq:frame-op-module-reform} will need to consider $\pi^* ((\sum_{i=1}^{k} \hs{\pi(w)g_i}{g_i})_{w\in \Subo})$. This is of no real consequence for the proofs. Hence we may conclude that for a multi-window Gabor frame $\GS(g_1 , \ldots , g_k ; \Sub)$ for $L^2 (G)$ with $g_1,\ldots , g_k \in S_0 (G)$ and associated (multi-window) frame operator $S$ we get $S^{-1}g_1, \ldots , S^{-1}g_k, S^{-1/2}g_1 , \ldots , S^{-1/2}g_k \in S_0 (G)$. Indeed one can go even further and do this for the matrix Gabor frames introduced in \cite{AuJaLu19}, which generalize multi-window super Gabor frames, using the setup from the same article. The key observation for doing this is that since $\ell^1 (\Subo,\cbar)$ is spectrally invariant in $\mathbb{B}(L^2 (G))$ we also have that $M_n (\ell^1 (\Subo,\cbar))$ is spectrally invariant in $M_n (\mathbb{B}(L^2 (G)))$ for any $n \in \N$ \cite{sc92}. 
\end{rmk}

\section{Acknowledgements}
The author would like to thank Michael Leinert for valuable feedback on an earlier draft of the article. 
Moreover, the author would like to thank Eduard Ortega for pointing out the concept of $C^*$-uniqueness, which broadened the scope and application of the article significantly. Lastly, the author also wishes to thank Eirik Berge, Mads S. Jakobsen, Franz Luef, and Eirik Skrettingland for valuable discussions during the development of the article.

\bibliography{main}
\bibliographystyle{abbrv}

\end{document}